\documentclass{article} 
\usepackage{hyperref,url}
\usepackage[numbers]{natbib}
\usepackage{mathrsfs} 
\usepackage{amsmath,amsbsy,amsfonts,amssymb,amsthm,dsfont,units}
\usepackage{algorithm,algorithmic}
\usepackage{mathtools}
\usepackage{color,cases}
\usepackage{tikz}
\usetikzlibrary{calc,shapes}
\usepackage{caption}
\usepackage{subcaption}
\usepackage[utf8]{inputenc}
\usepackage[english]{babel}
\usepackage{multirow}
\usepackage{bbm}
\usepackage{xr}
\usepackage{fullpage}
\usepackage{graphicx}
\usepackage{enumitem}
\usepackage{titling}

\def\norm#1{\mathopen\| #1 \mathclose\|}

\newcommand{\poly}{\mathop{\mbox{\rm poly}}}

\newcommand{\ignore}[1]{}

\def\reals{{\mathbb R}}

\newcommand\Zcal{\mathcal{Z}}

\def\bold0{\mathbf{0}}

\def\bB{\mathbf{B}}

\def\bA{\mathbf{A}}

\def\bI{\mathbf{I}}

\def\bB{{B}}
\def\bA{\mathbf{A}}
\def\bB{\mathbf{B}}
\def\bD{\mathbf{D}}

\def\bI{\mathbf{I}}

\def\bM{\mathbf{M}}

\def\eps{\varepsilon}
\def\epsilon{\varepsilon}

\newcommand{\defeq}{\stackrel{\text{def}}{=}}

\newcommand{\braces}[1]{\left\{#1\right\}}
\newcommand{\pa}[1]{\left(#1\right)}

\newcommand{\bra}[1]{\left[#1\right]}
\newcommand{\abs}[1]{\left|#1\right|}

\DeclareMathOperator{\argmin}{argmin}

\newtheorem{theorem}{Theorem}[section]
\newtheorem{corollary}[theorem]{Corollary}

\newtheorem{lemma}[theorem]{Lemma}
\newtheorem{fact}[theorem]{Fact}

\newtheorem{remark}[theorem]{Remark}

\theoremstyle{definition}

\DeclareMathOperator{\diag}{diag}

\newcommand\inner[1]{\langle #1 \rangle}

\newcommand\grad{\nabla}
\newcommand\hess{\nabla^2}
\newcommand\third{\nabla^3}
\newcommand\fourth{\nabla^4}
\newcommand\nab[1]{\nabla^#1}

\def\scalar{c}
\def\tA{\tilde{\mathbf{A}}}
\def\tQ{\tilde{\mathbf{Q}}}

\def\tb{\tilde{b}}
\def\tq{\tilde{q}}
\def\apm{\mu}
\def\linfty{\ell_\infty}

\DeclareMathOperator{\smax}{smax}
\DeclareMathOperator{\shinge}{shinge}
\DeclareMathOperator{\sabs}{sabs}
\DeclareMathOperator{\softlone}{soft-\ell{1}}

\DeclareMathOperator{\ssvm}{softSVM}
\DeclareMathOperator{\svm}{SVM}

\newcommand{\fastq}{\mathsf{FastQuartic}}

\title{Higher-Order Accelerated Methods for Faster Non-Smooth Optimization}
\author{
Brian Bullins$^*$
\and
Richard Peng$^\dagger$ 
}
\begin{document}

    \maketitle
\begin{abstract}
We provide improved convergence rates for various \emph{non-smooth} optimization problems via higher-order accelerated methods. In the case of $\ell_\infty$ regression, we achieves an $O(\eps^{-4/5})$ iteration complexity, breaking the $O(\eps^{-1})$ barrier so far present for previous methods. We arrive at a similar rate for the problem of $\ell_1$-SVM, going beyond what is attainable by first-order methods with prox-oracle access for non-smooth non-strongly convex problems. We further show how to achieve even faster rates by introducing higher-order regularization.

Our results rely on recent advances in near-optimal accelerated methods for higher-order smooth convex optimization. In particular, we extend Nesterov's smoothing technique to show that the standard softmax approximation is not only smooth in the usual sense, but also \emph{higher-order} smooth. With this observation in hand, we provide the first example of higher-order acceleration techniques yielding faster rates for \emph{non-smooth} optimization, to the best of our knowledge.
\end{abstract}

\footnotetext[1]{Princeton University and Google AI Princeton. \texttt{bbullins@cs.princeton.edu}.}
\footnotetext[2]{Georgia Tech and Microsoft Research Redmond. \texttt{rpeng@cc.gatech.edu}.}

\section{Introduction}
The benefit of smoothness for obtaining faster convergence has been well established in the optimization literature. Sadly, many machine learning tasks are inherently non-smooth, and thus do not inherit these favorable guarantees. In the non-smooth setting, it is known that one can achieve better than the black-box $O(1/\eps^2)$ rate for certain structured functions \citep{nesterov2005smooth}, including several (such as hinge loss, $\ell_1$ regression, etc.) that play a pivotal role in modern machine learning. 

In this paper, we are interested in developing faster methods for these important non-smooth optimization problems, one such example being the classic problem of $\linfty$ regression. As noted in \cite{ene2019improved}, even achieving a linear dependence in $\eps^{-1}$ has required careful handling of accelerated techniques for non-smooth optimization \citep{nesterov2005smooth, sherman2017area, sidford2018coordinate}. In this work, we show how to go \emph{beyond} these rates to achieve an iteration complexity that is \emph{sublinear} in $\eps^{-1}$. We further extend these results to the setting of soft-margin SVM, under various choices of regularization, again achieving iteration complexities that are sublinear in $\eps^{-1}$. Additionally, by making use of efficient tensor methods \cite{nesterov2018implementable, bullins2018fast}, we establish overall computational complexity in terms of (per-iteration) linear system solves, thus providing results that may be compared with \cite{christiano2011electrical, chin2013runtime, ene2019improved}.

The key observation of this work is that the softmax approximation to the max function, which we denote as $\smax_\apm(\cdot)$ (parameterized by $\apm > 0$), is not only smooth (i.e., its gradient is Lipschitz), but also \emph{higher-order smooth}. In particular, we establish Lipschitz continuity of its \emph{third} derivative by ensuring a bound on its fourth derivative, with Lipschitz constant $O(1/\apm^3)$. By combining this observation with recent advances in higher-order acceleration \citep{gasnikovdec2018global, jiang2018optimal, bubeck2018near, bullins2018fast}, we achieve an improved iteration complexity of $O(1/\eps^{4/5})$, thus going beyond the previous $O(1/\eps)$ dependence \citep{nesterov2005smooth, sherman2017area, sidford2018coordinate, ene2019improved}.

After bringing together the higher-order smoothness of softmax with near-optimal higher-order acceleration techniques, we arrive at the following results, beginning with $\linfty$ regression.
\begin{theorem}\label{thm:mainlinfty}
Let $f(x) = \norm{\bA x-b}_\infty$ for $b \in \reals^m$, $\bA \in \reals^{m \times d}$ s.t. $\bA^\top \bA \succ 0$, and let $x^* \defeq \argmin\limits_{x\in\reals^d}f(x)$. There is a method, initialized with $x_0$, that outputs $x_N$ such that
\begin{equation*}
f(x_N) - f(x^*) \leq \eps
\end{equation*}
in $O\pa{\frac{\log^{3/5}(m)\norm{x_0-x^*}_{\bA^\top \bA}^{4/5}}{\eps^{4/5}}}$ iterations, where each iteration requires $O(\log^{O(1)}(\Zcal/\eps))$ calls to a gradient oracle and solutions to linear systems of the form $\bA^\top \bD_x\bA\phi = w_x$, for diagonal matrix $\bD_x\in \reals^{m\times m}$, $w_x \in \reals^m$, and for some problem-dependent parameter $\Zcal$.
\end{theorem}

Our results are also applicable to
soft-margin SVMs, and so in particular, we get the following for
\mbox{$\ell_{1}$-SVM}~\cite{bradley1998feature, zhu2004onenorm, mangasarian2006exact}.
\begin{theorem}\label{thm:mainl1svm}
Let $f(x) = \lambda \norm{x}_1 + \frac{1}{m}\sum\limits_{i=1}^m \max\braces{0, 1-b_i\inner{a_i, x}}$ where $a_i \in \reals^d$, $b_i \in \reals$ for $i \in \bra{m}$, let $\tQ \defeq \bra{b_1a_1\ b_2a_2\ \dots\ b_ma_m}^\top$, and let $x^* \defeq \argmin\limits_{x\in\reals^d} f(x)$. There is a method, initialized with $x_0$, that outputs $x_N$ such that
\begin{equation*}
f(x_N) - f(x^*) \leq \eps
\end{equation*}
in $O\pa{\frac{\pa{\lambda d}^{3/5}(\lambda d + \norm{\tQ^\top \tQ}^2)^{1/5}\norm{x_0-x^*}^{4/5}}{\eps^{4/5}}}$ iterations, where each iteration requires $O(\log^{O(1)}(\Zcal/\eps))$ calls to a gradient oracle and linear system solver, for some problem-dependent parameter $\Zcal$.
\end{theorem}

We emphasize that such rates were not known before, to the best of our knowledge. Furthermore, our stronger oracle model seems necessary for going beyond an $O(1/\eps)$ dependence due to tight upper and lower bounds known for first-order methods with prox-oracle access, when the convex function is neither smooth nor strongly convex \cite{woodworth2016tight}. In addition, it is well-known that some structured linear systems can be solved in nearly-linear time \cite{spielman2004nearly, koutis2012fast}, making the per-iteration complexity competitive with first-order methods in such settings.

We also remark that determining the precise iteration complexities attainable under various higher-order oracle models and smoothness assumptions has been an incredibly active area of research \citep{nesterov2006cubic, nesterov2008accelerating, baes2009estimate, monteiro2013accelerated, agarwal2017finding, gasnikovdec2018global, jiang2018optimal, bubeck2018near, bullins2018fast}, and so our results complement these by extending their reach to \emph{non-smooth} problems under higher-order oracle access.

\subsection{Related work}
\paragraph{Smooth approximation techniques:} It was shown by Nesterov~\cite{nesterov2005smooth} that one can go beyond the black-box convergence of $O(1/\eps^2)$ to achieve an $O(1/\eps)$ rate for certain classes of non-smooth functions. The main idea was to carefully smooth the well-structured function, and the work goes on to present several applications of the method, including $\linfty$ and $\ell_1$ regression, in addition to saddle-point games. However, the methods for all of these examples incur an $O(1/\eps)$ dependence which remains in several works that build upon these techniques \cite{sherman2017area, sidford2018coordinate}. For a more comprehensive overview, we refer the reader to \cite{beck2012smoothing}.

\paragraph{Higher-order accelerated methods:} Several works have considered accelerated variants of optimization methods based on access to higher-order derivative information. Nesterov~\cite{nesterov2008accelerating} showed that one can accelerate cubic regularization, under a Lipschitz Hessian condition, to attain faster convergence, and these results were later generalized by Baes~\cite{baes2009estimate} to arbitrary higher-order oracle access under the appropriate notions of (higher-order) smoothness. The rate attained in \cite{nesterov2008accelerating} was further improved upon by Monteiro and Svaiter~\cite{monteiro2013accelerated}, and lower bounds have established that the oracle complexity of this result is nearly tight (up to logarithmic factors) when the Hessian is Lipschitz \cite{arjevani2018oracle}. Until recently, however, it was an open question whether these lower bounds are tight for general higher-order oracle access (and smoothness), though this question has been mostly resolved as a result of several works developed over the past year \cite{gasnikovdec2018global, jiang2018optimal, bubeck2018near, bullins2018fast}.

\paragraph{$\linfty$ regression:} Various regression problems play a central role in numerous computational and machine learning tasks. Designing better methods for $\linfty$ regression in particular has led to faster approximate max flow algorithms \cite{christiano2011electrical, chin2013runtime, kelner2014almost, sherman2017area, sidford2018coordinate}. Recently, Ene and Vladu \cite{ene2019improved} presented a method for $\ell_\infty$ regression, based on iteratively reweighted
least squares, that achieves an iteration complexity of $O(m^{1/3}\log(1/\eps)/\eps + \log(m/\eps)/\eps^2)$. We note that their rate of convergence has an $O(m^{1/3})$ dependence, whereas our result, in contrast, includes a diameter term, i.e., $\norm{x_0-x^*}^{4/5}$.

\paragraph{Soft-margin SVM:} Support vector machines (SVMs) \cite{cortes1995support} have enjoyed widespread adoption for classification tasks in machine learning \cite{cristianini2000introduction}. For the soft-margin version, several approaches have been proposed for dealing with the non-smooth nature of the hinge loss. The standard approach is to cast the ($\ell_2$-regularized) SVM problem as a quadratic programming problem \cite{platt1998sequential, boyd2004convex}. Stochastic sub-gradient methods have also been successful due to their advantage in per-iteration cost \cite{shalev2011pegasos}. While $\ell_2$-SVM is arguably the most well-known variant, $\ell_p$-SVMs, for general $p \geq 1$, have also been studied \cite{bradley1998feature}. $\ell_1$-SVMs \cite{zhu2004onenorm, mangasarian2006exact} are appealing, in particular, due to their sparcity-inducing tendencies, though they forfeit the strong convexity guarantees that come with $\ell_2$ regularization \cite{allen2016optimal}. 

\paragraph{Interior-point methods:} It is well-known that both $\linfty$ regression and $\ell_1$-SVM can be expressed as linear programs \cite{boyd2004convex, bradley1998feature}, and thus are amenable to fast LP solvers \cite{LeeS14, cohen2018solving}. In particular, this means that each can be solved in either $\tilde{O}(d^{\omega})$ time (where $\omega \sim 2.373$ is the matrix multiplication constant) \cite{cohen2018solving}, or in $\tilde{O}(\sqrt{d})$ linear system solves \cite{LeeS14}. We note that, while these methods dominate in the low-error regime, our method is competitive, under modest choices of $\eps$ and favorable linear system solves, when the diameter term $\norm{x_0-x^*}^{4/5} \leq O(\sqrt{d})$ (up to logarithmic factors).

\subsection{Our contributions}
The main contributions of this work are as follows:
\begin{enumerate}
	\item We provide improved higher-order oracle complexity for several important \emph{non-smooth} optimization problems, by combining near-optimal higher-order acceleration with the appropriate smooth approximations.
	\item By leveraging efficient tensor methods \cite{nesterov2018implementable, bullins2018fast}, we go beyond the oracle model to establish overall computational complexity for these non-smooth problems that, for certain parameter regimes, improves upon previous results.
\end{enumerate}

We further stress that the convergence guarantees presented in this work surpass the tight upper and lower bounds known under first-order and prox-oracle access, for non-smooth and non-strongly convex functions \cite{woodworth2016tight}. Thus, we observe that higher-order oracle access provides an advantage not only for functions that are sufficiently smooth, but also in the \emph{non-smooth} setting.

In addition, we wish to note the importance of relying on more recent advances in near-optimal higher-order acceleration \citep{gasnikovdec2018global, jiang2018optimal, bubeck2018near, bullins2018fast}. We may recall in particular that the higher-order acceleration scheme in \cite{baes2009estimate} achieves a rate of $O\pa{\pa{L_p/\eps}^{1/(p+1)}}$ (assuming $p^{th}$ derivative is $L_p$-Lipschitz). Thus, for the case of $p=3$ (whereby $L_3 \approx 1/\eps^3$), this approach would not improve upon the previous $O(1/\eps)$ dependence since, roughly speaking, we would only expect to recover a rate of $O\pa{\pa{1/\eps^{3+1}}^{1/4}} = O(1/\eps)$.

While one may also consider repeatedly applying Gaussian smoothing to induce higher-order smoothness, this approach suffers from two primary drawbacks: (1) a straightforward application would incur an additional $O(\poly(d))$ term, and (2) it would become necessary to compute higher-order derivatives of the Gaussian-smoothed function.

\section{Setup}
Let $u, v$ denote vectors in $\reals^d$. Throughout, we let $v_i$ denote the $i$-th coordinate of $v$, and we let $\bra{k} \defeq \braces{1, \dots, k}$ for $k \geq 1$. We let $u \circ v$ denote the Hadamard product, i.e., $(u \circ v)_i = u_iv_i$ for all $i \in [d]$. Furthermore, we will define $v^2 \defeq v \circ v$ and $v^3 \defeq v \circ v \circ v$. We let $\Delta_m \defeq \braces{x \in \reals^m : \sum_i x_i = 1, x_i \geq 0}$ denote the $m$-dimensional simplex. We let $\norm{v}_p$ denote the standard $\ell_p$ norm, and we drop the subscript to let $\norm{\cdot}$ denote the $\ell_2$ norm. Let $\bB \in \reals^{d\times d}$ be a symmetric positive-definite matrix, i.e., $\bB \succ 0$. Then, we may define the matrix-induced norm of $v$ (w.r.t. $\bB$) as $\norm{v}_\bB \defeq \sqrt{v^\top \bB v}$, and we let $\norm{\bB} \defeq \lambda_{\max}(\bB)$.

We now make formal a (higher-order) notion of smoothness. Specifically, for $p \geq 1$, we say a $p$-times differentiable function $f(\cdot)$ is \emph{$L_p$-smooth (of order $p$)} w.r.t. $\norm{\cdot}_\bB$ if the $p^{th}$ derivative is $L_p$-Lipschitz continuous, i.e., for all $x, y \in \reals^d$,
\begin{equation}\label{eq:highersmooth}
    \norm{\nab{p} f(y) - \nab{p} f(x)}_{\bB}^* \leq L_p\norm{y - x}_\bB,
\end{equation}
where we define
\begin{equation*}
\norm{\nab{p} f(y) - \nab{p} f(x)}_{\bB}^* \defeq \max\limits_{h : \norm{h}_\bB \leq 1} \Bigl|\nab{p} f(y) [h]^{p} - \nab{p} f(x) [h]^{p}\Bigr| \ ,\end{equation*}
and where 
\begin{equation*}
\nab{p} f(x) [h]^{p} \defeq \nab{p} f(x) \underbrace{[h, h, \dots, h]}_{p\ \text{times}}.
\end{equation*}

Observe that, for $p = 1$, this recovers the usual notion of smoothness, and so our convention will be to refer to first-order smooth functions as simply smooth. A complementary notion is that of strong convexity, and its higher-order generalization known as uniform convexity \cite{nesterov2008accelerating}. In particular, $f(\cdot)$ is \emph{$\sigma_p$-uniformly convex (of order $p$)} with respect to $\norm{\cdot}_\bB$ if, for all $x, y \in \reals^d$,
\begin{equation*}
f(y) \geq f(x) + \inner{\grad f(x), y-x} + \frac{\sigma_p}{p}\norm{y-x}_\bB^p.
\end{equation*} Again, we may see that this captures the typical $\sigma_2$-strong convexity (w.r.t. $\norm{\cdot}_\bB$) by setting $p = 2$.

\section{Softmax approximation and $\linfty$ regression}
We recall from \cite{nesterov2005smooth, sidford2018coordinate} the standard softmax approximation, for $x \in \reals^m$:
\begin{equation}\label{eq:smaxdef}
\smax_\apm(x) \defeq \apm\log\pa{\sum\limits_{i=1}^m e^{\frac{x_i}{\apm}}}.
\end{equation}
It is straightforward to observe that \eqref{eq:smaxdef} is $\frac{1}{\apm}$-smooth, and furthermore that it smoothly approximates the max function, i.e., $\max_{j \in [m]} x_j$.
\begin{fact}\label{fact:smaxapprox}
For all $x \in \reals^m$,
\begin{equation}
\max\limits_{j \in \bra{m}} x_j \leq \smax_\apm(x) \leq \apm\log(m) + \max\limits_{j \in \bra{m}} x_j.
\end{equation}
\end{fact}

Note that this approximation can be used for $\norm{x}_\infty$, since $\norm{x}_\infty = \max\limits_{j \in \bra{m}} \abs{x_j}$, and $\abs{x_j} = \max\braces{x_j, -x_j}$.
It follows that we may determine a smooth approximation of $\linfty$ regression, i.e.,
\begin{equation}\label{eq:linftyreg}
\min\limits_{x \in \reals^d} \norm{\tA x - \tb}_\infty, \quad \tA \in \reals^{m \times d}, \ \ \tb \in \reals^m,
\end{equation}
as $\smax_\apm(\bA x-b)$, where $\bA = \begin{pmatrix} \tA \\ -\tA \end{pmatrix}$ and $b = \begin{pmatrix} \tb \\ -\tb \end{pmatrix}$. 

Having now formalized the connection between $\smax_\apm(\cdot)$ and $\norm{\cdot}_\infty$, we assume throughout the rest of the paper that $\bA \in \reals^{m \times d}$ and $b \in \reals^m$, as the difference in dimension between $\tA$, $\tb$ and $\bA$, $b$ only affects the final convergence by a constant factor. In addition, we will assume that $\bA$ is such that $\bA^\top \bA \succ 0$, and thus we consider the regime where $m \geq d$.

\subsection{Softmax calculus}
To simplify notation, we let $Z_\apm(x) = \sum\limits_{i=1}^m e^{\frac{x_i}{\apm}}$, and so $\smax_\apm(x) = \apm\log\pa{Z_\apm(x)}$. Note that we have
\begin{equation}
\grad \smax_\apm(x)_i = \frac{e^{\frac{x_i}{\apm}}}{Z_\apm(x)}, \quad i \in \bra{m}.
\end{equation}
Furthermore, since $\grad \smax_\apm(x) \in \Delta_m$ for all $x \in \reals^m$, it follows that, for all $p \geq 1$,
\begin{equation}\label{eq:gradnormineq}
\norm{\grad \smax_\apm(x)}_p \leq 1.
\end{equation}
We may also see that
\begin{equation}\label{eq:smaxhess}
\hess \smax_\apm(x) = \frac{1}{\apm}\pa{\diag(\grad \smax_\apm(x)) - \grad \smax_\apm(x)\grad \smax_\apm(x)^\top}.
\end{equation}
Since $\hess \smax_\apm(x)$ is a symmetric bilinear form for all $x \in \reals^m$, it follows that, for all $h_1, h_2 \in \reals^m$,
\begin{equation}\label{eq:smaxhessh1h2}
\hess \smax_\apm(x)[h_1,h_2] = \frac{1}{\apm}\pa{\inner{\grad \smax_\apm(x), h_1\circ h_2} - \inner{\grad \smax_\apm(x), h_1}\cdot \inner{\grad \smax_\apm(x), h_2}}.
\end{equation}

\subsection{Higher-order smoothness}
As mentioned previously, one of the key observations of this work is that softmax is equipped with favorable higher-order smoothness properties. We begin by showing a bound on its fourth derivative, as established by the following lemma, and we provide its proof in the appendix.

\begin{lemma}\label{lem:smaxfourthbound}
For all $x$, $h \in \reals^d$,
\begin{equation}\label{eq:smaxfourthbound}
\abs{\fourth \smax_\apm(x)[h,h,h,h]} \leq \frac{15}{\apm^3}\norm{h}_2^4.
\end{equation}
\end{lemma}

It will also be helpful to note the following standard result on how a bound on the fourth derivative implies Lipschitz-continuity of the third derivative.
\begin{lemma}\label{lem:lipschitz}
Let $f(\cdot)$ be a $4$-times differentiable function, let $L_3 > 0$ and $\bA$ be such that $\bA^\top \bA \succ 0$, and suppose, for all $\zeta, h \in \reals^d$,
\begin{equation}\label{eq:fourthboundsup}
\abs{\fourth f(\zeta)[h,h,h,h]} \leq L_3\norm{\bA h}_2^4.
\end{equation}
Then we have that, for all $x, y \in \reals^d$,
\begin{equation}
\norm{\third f(y)-\third f(x)}_{\bA^\top \bA}^* \leq L_3\norm{y-x}_{\bA^\top \bA}.
\end{equation}
\end{lemma}

Having determined these bounds, we now provide smoothness guarantees for the softmax approximation to $\linfty$ regression.
\begin{theorem}\label{thm:mainlipschitz}
Let $f(x) = \smax_\apm(\bA x-b)$. Then, $f(x)$ is (order 3) $\frac{15}{\apm^3}$-smooth w.r.t. $\norm{\cdot}_{\bA^\top \bA}$.
\end{theorem}

\section{Higher-order acceleration}
We now rely on recent techniques for near-optimal higher-order acceleration \citep{gasnikovdec2018global, jiang2018optimal, bubeck2018near, bullins2018fast}.
For these higher-order iterative methods, assuming $f(\cdot)$ is (order $p$) $L_p$-smooth, the basic idea for each iteration is to determine a minimizer of the subproblem given by the $p^{th}$-order Taylor expansion $(p \geq 1)$, centered around the current iterate $x_t$,
 plus a $(p+1)^{th}$-order regularization term, i.e.,
 \begin{equation}
     x_{t+1} = \argmin\limits_{x \in \reals^d} \Omega_{x_t, p, \bB}(x),
 \end{equation}
 where, for all $x, y \in \reals^d$,
 \begin{equation}\label{eq:regmodel}
\Omega_{x, p, \bB}(y) \defeq f(x) + \sum\limits_{i=1}^p \frac{1}{i!}\nab{i} f(x)[y-x]^i + \frac{2pL_p}{(p+1)!}\norm{y-x}_\bB^{p+1}.
\end{equation}
 
 Given access to such an oracle, it is possible to combine it with a carefully-tuned accelerated scheme to achieve an improved iteration complexity when the $p^{th}$ derivative is Lipschitz. In contrast to the higher-order accelerated scheme of Nesterov \cite{nesterov2008accelerating} (later generalized by Baes \cite{baes2009estimate}), these near-optimal rates rely on a certain additional binary search procedure, as first observed by Monteiro and Svaiter \cite{monteiro2013accelerated}.

In particular, we are motivated by the $\fastq$ method \citep{bullins2018fast}, whereby we provide a sketch of the algorithm here. Note that, for the sake of clarity, various approximations found in the precise algorithm have been omitted, and we refer the reader to \cite{bullins2018fast} for the complete presentation.

\begin{algorithm}[h]
	\caption{$\fastq$ (Sketch)}

	\begin{algorithmic}\label{alg:fastq}
		\STATE {\bfseries Input:} $x_0 = 0$,  $A_0 = 0$, $\bB \succ 0$, $N$.
		\STATE Define $\psi_0(x) \defeq \frac{1}{2}\norm{x - x_0}_\bB^2$.
		\FOR{$k=0$ {\bfseries to} $N-1$}
        \STATE $v_k = \argmin\limits_{x \in \reals^d} \psi_k(x)$
        \STATE Find $\rho_k > 0$, $x_{k+1} \in \reals^d$ such that $\rho_k \approx \norm{x_{k+1} - y_k}_{\bB}^2$, where:
        \begin{align*}
        a_{k+1} &= \frac{1 + \sqrt{1+4L_3A_k\rho_k}}{2L_3\rho_k} \quad\quad\quad\pa{\implies \pa{{a_{k+1}}}^2 = \frac{A_k + a_{k+1}}{L_3 \rho_k}}\\
        A_{k+1} &= A_k + a_{k+1}, \qquad \tau_k = \frac{a_{k+1}}{A_{k+1}}, \qquad y_k = (1-\tau_k)x_k + \tau_k v_k\\
        x_{k+1} &= \argmin\limits_{x\in\reals^d} \Omega_{y_k, 3, \bB}(x)\qquad \text{(As defined in eq.\eqref{eq:regmodel}.)}
        \end{align*}\vspace{-0.3cm}
		\STATE $\psi_{k+1} = \psi_k + a_{k+1}\bra{f(x_{k+1}) + \inner{\grad f(x_{k+1}), x - x_{k+1}}}$
		\ENDFOR
		\RETURN $x_{N}$
	\end{algorithmic}
\end{algorithm}

As established by Bullins \cite{bullins2018fast}, $\fastq$ provides us with the following guarantee.

\begin{theorem}[\cite{bullins2018fast}, Theorem 4.1]\label{thm:smoothrate} Suppose $f(x)$ is (order 3) $L_3$-smooth w.r.t. $\norm{\cdot}_\bB$ for $\bB \succ 0$. Then, $\fastq$ finds a point $x_N$ such that
\begin{equation*}
f(x_N) - f(x^*) \leq \eps
\end{equation*}
in $O\pa{\pa{\frac{L_3 \norm{x_0-x^*}_\bB^4}{\eps}}^{1/5}}$ iterations, where each iteration requires $O(\log^{O(1)}(\Zcal/\eps))$ calls to a gradient oracle and linear system solver, and where $\Zcal$ is a polynomial in various problem-dependent parameters.
\end{theorem}

Given this result, we have the following corollary which will be useful for our smoothed minimization problem.

\begin{corollary}\label{cor:mainlinfty}
Let $f_\apm(x) = \smax_\apm(\bA x-b)$ be the softmax approximation to \eqref{eq:linftyreg} for $\apm = \frac{\eps}{2\log(m)}$, where $\bA$ is such that $\bA^\top \bA \succ 0$. Then, letting $x_\apm^* \defeq \argmin\limits_{x \in \reals^d} f_\apm(x)$, $\fastq$ finds a point $x_N$ such that
\begin{equation*}
f_\apm(x_N) - f_\apm(x_\apm^*) \leq \frac{\eps}{2}
\end{equation*}
in $O\pa{\frac{\log^{3/5}(m)\norm{x_0-x^*}_{\bA^\top \bA}^{4/5}}{\eps^{4/5}}}$ iterations, where each iteration requires $O(\log^{O(1)}(\Zcal/\eps))$ calls to a gradient oracle and solutions to linear systems of the form $\bA^\top \bD_x\bA\phi = w_x$, for diagonal matrix $\bD_x \in \reals^{m \times m}$ and $w_x \in \reals^d$.
\end{corollary}

We are now equipped with the tools necessary for proving Theorem \ref{thm:mainlinfty}.
\begin{proof}[Proof of Theorem \ref{thm:mainlinfty}] The proof follows simply by combining Fact \ref{fact:smaxapprox}, for $\apm = \frac{\eps}{2\log(m)}$, with Corollary \ref{cor:mainlinfty}.
\end{proof}

\section{Soft-margin SVM}
In this section we shift our focus to consider various instances of soft-margin SVM.
It is known that in the $\ell_2$ case, an improved rate of $O(1/\eps^{1/2})$ is
possible \cite{orabona2012prisma, nesterov2013gradient, allen2016optimal}, and so we give the first sub-$O(1/\eps)$ rate for variants of
SVM that are both non-smooth \emph{and} non-strongly convex. In Section \ref{sec:l1svm},
we handle $\ell_1$ regularization, and in Section \ref{sec:l4svm}, we consider the case of higher-order regularizers.

\subsection{$\ell_1$-regularized SVM}\label{sec:l1svm}
We begin with $\ell_1$-regularized soft-margin SVM ($\ell_1$-SVM), i.e.,
\begin{equation}\label{eq:l1svmfunc}
f(x) = \lambda\norm{x}_1 + \frac{1}{m}\sum\limits_{i=1}^m \max\braces{0, 1-b_i\inner{a_i, x}},
\end{equation}
for $a_i \in \reals^d$, $b_i \in \reals$ ($i\in [m]$), and $\lambda > 0$. To simplify the notation, we define $\svm(x)~\defeq~\frac{1}{m}\sum\limits_{i=1}^m \max\braces{0, 1-x_i}$. Letting $\tq_i \defeq b_ia_i$ and 
\begin{equation}\label{eq:tqdef}
\tQ \defeq \bra{\tq_1\ \tq_2\ \dots\  \tq_m}^\top,
\end{equation}
 we may then rewrite $f(x) = \lambda\norm{x}_1 + \svm(\tQ x)$. We now make the following observations concerning softmax-based approximations for $\norm{\cdot}_1$ and $\max\braces{0, \cdot}$.

\begin{lemma}[$\ell_1$ approximation]\label{lem:l1approx} Let $\sabs_\apm(\scalar) \defeq \smax_\apm([\scalar,-\scalar])$ for $\scalar \in \reals$, and let $\softlone_\apm(x) \defeq \sum\limits_{i=1}^m \sabs_\apm(x_i)$ for $x \in \reals^m$. Then, we have that
\begin{equation}
\norm{x}_1 \leq \softlone_\apm(x) \leq \norm{x}_1 + \apm m.
\end{equation}
\end{lemma}

\begin{lemma}[Smooth hinge loss approximation]\label{lem:hingeapprox} Let $\shinge_\apm(\scalar) \defeq \smax_\apm([0, \scalar])$ for $\scalar \in \reals$. Then
\begin{equation}
\max\braces{0,\scalar} \leq \shinge_\apm(\scalar) \leq \max\braces{0,\scalar}+\apm.
\end{equation}
\end{lemma}

This gives us a natural smooth approximation to $\svm(x)$, namely,
\begin{equation}
\ssvm_\apm(x) \defeq \frac{1}{m}\sum\limits_{i=1}^m \shinge_\apm(1 - x_i).
\end{equation}
Taken together with these approximations, we arrive at the following lemma, the proof of which follows by combining Lemmas \ref{lem:l1approx} and \ref{lem:hingeapprox}.
\begin{lemma}\label{lem:l1ssvmapprox}
Let $f_\apm(x) = \lambda\softlone_\apm(x) + \ssvm(\tQ x)$, and let $f(x)$ be as in \eqref{eq:l1svmfunc}. Then, for all $x \in \reals^d$,
\begin{equation}
f(x) \leq f_\apm(x) \leq f(x) + 2\apm \lambda d.
\end{equation}
\end{lemma}

As was the case for $\linfty$ regression, in order to make use of the guarantees provided by $\fastq$, we must first show higher-order smoothness, and so we have following theorem.
\begin{theorem}\label{thm:mainlipschitzsvm}
Let $f_\apm(x) = \lambda\softlone_\apm(x) + \ssvm_\apm(\tQ x)$. Then, $f_\apm(x)$ is (order 3) $L_3$-smooth w.r.t. $\norm{\cdot}_2$, for $L_3 = \frac{15\pa{\lambda d + \norm{\tQ^\top \tQ}^2}}{\apm^3}$.
\end{theorem}

\begin{corollary}\label{cor:mainl1svm}
Let $f_\apm(x) = \lambda\softlone_\apm(x) + \ssvm_\apm(\tQ x)$ be the smooth approximation to $f(x)$ (as in \eqref{eq:l1svmfunc}) with $\apm = \frac{\eps}{4\lambda d}$ for $\eps > 0$. Then, letting $x_\apm^* \defeq \argmin\limits_{x \in \reals^d} f_\apm(x)$, $\fastq$ finds a point $x_N$ such that
\begin{equation*}
f_\apm(x_N) - f_\apm(x_\apm^*) \leq \frac{\eps}{2}
\end{equation*}
in $O\pa{\frac{\pa{\lambda d}^{3/5}(\lambda d + \norm{\tQ^\top \tQ}^2)^{1/5}\norm{x_0-x^*}^{4/5}}{\eps^{4/5}}}$ iterations, where each iteration requires $O(\log^{O(1)}(\Zcal/\eps))$ calls to a gradient oracle and linear system solver, and where $\Zcal$ is a polynomial in various problem-dependent parameters.
\end{corollary}
\begin{proof}
The corollary follows by Theorem \ref{thm:smoothrate}, using the smoothness guarantee from Theorem \ref{thm:mainlipschitzsvm}.
\end{proof}
\noindent The proof of Theorem \ref{thm:mainl1svm} then follows by combining Lemma \ref{lem:l1ssvmapprox}, for $\apm = \frac{\eps}{4\lambda d}$, with Corollary \ref{cor:mainl1svm}.

\subsection{Higher-order regularization}\label{sec:l4svm}
The soft-margin SVM model has been studied with various choices of regularization beyond $\ell_1$ and $\ell_2$ \cite{bradley1998feature}. Just as introducing strong convexity can lead to faster convergence for $\ell_2$-SVM \cite{allen2016optimal}, we may see that a similar advantage may be obtained for an appropriately chosen regularizer that is uniformly convex. More concretely, if we consider the $\ell_p$-regularized soft-margin SVM for $p = 4$, we are able to use the following theorem from \cite{bullins2018fast} which holds for functions that are both higher-order smooth and uniformly convex.

\begin{theorem}[\cite{bullins2018fast}, Theorem 4.2] \label{thm:kappa} Suppose $f(x)$ is (order 3) $L_3$-smooth and (order 4) $\sigma_4$-uniformly convex w.r.t. $\norm{\cdot}_\bB$, let $\kappa_4 \defeq \frac{L_3}{\sigma_4}$, and let $x^* \defeq \argmin\limits_{x \in \reals^d} f(x)$. Then, with the appropriate restarting procedure, $\fastq$ finds a point $x_N$ such that
\begin{equation*}
f(x_N) - f(x^*) \leq \eps
\end{equation*}
in $O\pa{\kappa_4^{1/5}\log\pa{\frac{f(x_0)-f(x^*)}{\eps}}}$ iterations, where each iteration requires $O(\log^{O(1)}(\Zcal/\eps))$ calls to a gradient oracle and linear system solver, and where $\Zcal$ is a polynomial in various problem-dependent parameters.
\end{theorem}

\begin{remark}\label{rem:quartic}
While the choice of $p = 4$ may appear arbitrary, we note that the fourth-order regularized Taylor models (eq.\eqref{eq:regmodel}, for $p=4$) permit efficiently computable (approximate) solutions \cite{nesterov2018implementable, bullins2018fast}, and developing efficient tensor methods for subproblems beyond the fourth-order model remains an interesting open problem.
\end{remark}

Thus, we may consider $\ell_4$-SVM, i.e., $f(x) = \lambda \norm{x}_4^4 + \svm(\tQ x)$ (as presented in \cite{bradley1998feature}), along with its smooth counterpart $f_\apm(x) = \lambda\norm{x}_4^4 + \ssvm_\apm(\tQ x)$, which brings us to the following corollary and theorem for $\ell_4$-SVM.
\begin{corollary}\label{cor:mainl4svm}
Let $f_\apm(x) = \lambda\norm{x}_4^4 + \ssvm_\apm(\tQ x)$ be the smooth approximation to $f(x)$ with $\apm = \frac{\eps}{4}$ for $\eps > 0$. Then, letting $x_\apm^* \defeq \argmin\limits_{x \in \reals^d} f_\apm(x)$, and with the appropriate restarting procedure, $\fastq$ finds a point $x_N$ such that
\begin{equation*}
f_\apm(x_N) - f_\apm(x_\apm^*) \leq \frac{\eps}{2}
\end{equation*}
in $O\pa{\frac{(d(\lambda + \norm{\tQ^\top \tQ}^2))^{1/5}}{\eps^{3/5}}\log\pa{\frac{f(x_0)-f(x^*)}{\eps}}}$ iterations, where each iteration requires $O(\log^{O(1)}(\Zcal/\eps))$ calls to a gradient oracle and linear system solver, and where $\Zcal$ is a polynomial in various problem-dependent parameters.
\end{corollary}
\begin{theorem}\label{thm:mainl4svm}
Let $f(x) = \lambda \norm{x}_4^4 + \frac{1}{m}\sum\limits_{i=1}^m \max\braces{0, 1-b_i\inner{a_i, x}}$ where $a_i \in \reals^d$, $b_i \in \reals\ $ for $i \in \bra{m}$, and let $x^* \defeq \argmin\limits_{x\in\reals^d} f(x)$. There is a method, initialized with $x_0$, that outputs $x_N$ such that
\begin{equation*}
f(x_N) - f(x^*) \leq \eps
\end{equation*}
in $O\pa{\frac{(d(\lambda + \norm{\tQ^\top \tQ}^2))^{1/5}}{\eps^{3/5}}\log\pa{\frac{f(x_0)-f(x^*)}{\eps}}}$ iterations, where each iteration requires $O(\log^{O(1)}(\Zcal/\eps))$ calls to a gradient oracle and linear system solver, for some problem-dependent parameter $\Zcal$.
\end{theorem}
\begin{proof}
The proof follows immediately from Lemma \ref{lem:hingeapprox}, for $\apm = \frac{\eps}{4}$, and Corollary \ref{cor:mainl4svm}.
\end{proof}

While we acknowledge that this result is limited to the (less common) case of $\ell_4$-SVM (see Remark \ref{rem:quartic}), we include it here to illustrate the iteration complexity improvement, from $O(1/\eps^{4/5})$ to $\tilde{O}(1/\eps^{3/5})$, under additional uniform convexity guarantees, similar to the improvement gained for strongly convex non-smooth problems \cite{beck2009fast, allen2016optimal}.

\section{Conclusion}
In this work, we have shown how to harness the power of higher-order acceleration for faster non-smooth optimization. While we have focused primarily on convex optimization, one potential direction would be to investigate if these techniques can extend to the non-smooth \emph{non-convex} setting. Although it is not possible in general to guarantee convergence to a first-order critical point (i.e., $\norm{\grad f(x)} \leq \eps$) for non-smooth problems, recent work has consider a relaxed version of the non-convex problem with a Moreau envelope-based smoothing \cite{davis2018complexity}. Improving max flow would be another interesting future direction, perhaps by connecting these higher-order techniques with the results in \cite{sherman2017area, sidford2018coordinate}.

\section*{Acknowledgements}
We thank S\'ebastien Bubeck, Yin Tat Lee, Sushant Sachdeva, Cyril Zhang, and Yi Zhang for numerous helpful conversations. BB is supported by Elad Hazan’s NSF grant CCF-1704860. RP is partially supported by the National Science Foundation under Grant \# 1718533.

\bibliography{main}

\begin{thebibliography}{10}

\bibitem{agarwal2017finding}
N.~Agarwal, Z.~Allen-Zhu, B.~Bullins, E.~Hazan, and T.~Ma.
\newblock Finding approximate local minima faster than gradient descent.
\newblock In {\em Proceedings of the 49th Annual ACM SIGACT Symposium on Theory
  of Computing}, pages 1195--1199. ACM, 2017.

\bibitem{allen2016optimal}
Z.~Allen-Zhu and E.~Hazan.
\newblock Optimal black-box reductions between optimization objectives.
\newblock In {\em Advances in Neural Information Processing Systems}, pages
  1614--1622, 2016.

\bibitem{arjevani2018oracle}
Y.~Arjevani, O.~Shamir, and R.~Shiff.
\newblock Oracle complexity of second-order methods for smooth convex
  optimization.
\newblock {\em Mathematical Programming}, pages 1--34, 2018.

\bibitem{baes2009estimate}
M.~Baes.
\newblock Estimate sequence methods: extensions and approximations.
\newblock 2009.

\bibitem{beck2009fast}
A.~Beck and M.~Teboulle.
\newblock A fast iterative shrinkage-thresholding algorithm for linear inverse
  problems.
\newblock {\em SIAM Journal on Imaging Sciences}, 2(1):183--202, 2009.

\bibitem{beck2012smoothing}
A.~Beck and M.~Teboulle.
\newblock Smoothing and first order methods: A unified framework.
\newblock {\em SIAM Journal on Optimization}, 22(2):557--580, 2012.

\bibitem{boyd2004convex}
S.~Boyd and L.~Vandenberghe.
\newblock {\em Convex optimization}.
\newblock Cambridge University Press, 2004.

\bibitem{bradley1998feature}
P.~S. Bradley and O.~Mangasarian.
\newblock Feature selection via concave minimization and support vector
  machines.
\newblock In {\em International Conference on Machine Learning}, pages 82--90,
  1998.

\bibitem{bubeck2018near}
S.~Bubeck, Q.~Jiang, Y.~T. Lee, Y.~Li, and A.~Sidford.
\newblock Near-optimal method for highly smooth convex optimization.
\newblock {\em arXiv preprint arXiv:1812.08026}, 2018.

\bibitem{bullins2018fast}
B.~Bullins.
\newblock Fast minimization of structured convex quartics.
\newblock {\em arXiv preprint arXiv:1812.10349}, 2018.

\bibitem{chin2013runtime}
H.~H. Chin, A.~Madry, G.~L. Miller, and R.~Peng.
\newblock Runtime guarantees for regression problems.
\newblock In {\em Proceedings of the 4th Conference on Innovations in
  Theoretical Computer Science}, pages 269--282. ACM, 2013.

\bibitem{christiano2011electrical}
P.~Christiano, J.~A. Kelner, A.~Madry, D.~A. Spielman, and S.-H. Teng.
\newblock Electrical flows, {L}aplacian systems, and faster approximation of
  maximum flow in undirected graphs.
\newblock In {\em Proceedings of the Forty-Third Annual ACM Symposium on Theory
  of Computing}, pages 273--282. ACM, 2011.

\bibitem{cohen2018solving}
M.~B. Cohen, Y.~T. Lee, and Z.~Song.
\newblock Solving linear programs in the current matrix multiplication time.
\newblock {\em arXiv preprint arXiv:1810.07896}, 2018.

\bibitem{cortes1995support}
C.~Cortes and V.~Vapnik.
\newblock Support-vector networks.
\newblock {\em Machine Learning}, 20(3):273--297, 1995.

\bibitem{cristianini2000introduction}
N.~Cristianini, J.~Shawe-Taylor, et~al.
\newblock {\em An introduction to support vector machines and other
  kernel-based learning methods}.
\newblock Cambridge university press, 2000.

\bibitem{davis2018complexity}
D.~Davis and D.~Drusvyatskiy.
\newblock Complexity of finding near-stationary points of convex functions
  stochastically.
\newblock {\em arXiv preprint arXiv:1802.08556}, 2018.

\bibitem{ene2019improved}
A.~Ene and A.~Vladu.
\newblock Improved convergence for $\ell_\infty$ and $\ell_1$ regression via
  iteratively reweighted least squares.
\newblock In {\em International Conference on Machine Learning}, 2019.

\bibitem{gasnikovdec2018global}
A.~Gasnikov, P.~Dvurechensky, E.~Gorbunov, D.~Kovalev, A.~Mohhamed,
  E.~Chernousova, and C.~A. Uribe.
\newblock The global rate of convergence for optimal tensor methods in smooth
  convex optimization.
\newblock {\em arXiv preprint arXiv:1809.00382 (v10)}, 2018.

\bibitem{jiang2018optimal}
B.~Jiang, H.~Wang, and S.~Zhang.
\newblock An optimal high-order tensor method for convex optimization.
\newblock {\em arXiv preprint arXiv:1812.06557}, 2018.

\bibitem{kelner2014almost}
J.~A. Kelner, Y.~T. Lee, L.~Orecchia, and A.~Sidford.
\newblock An almost-linear-time algorithm for approximate max flow in
  undirected graphs, and its multicommodity generalizations.
\newblock In {\em Proceedings of the twenty-fifth annual ACM-SIAM symposium on
  Discrete algorithms}, pages 217--226. SIAM, 2014.

\bibitem{koutis2012fast}
I.~Koutis, G.~L. Miller, and R.~Peng.
\newblock A fast solver for a class of linear systems.
\newblock {\em Communications of the ACM}, 55(10):99--107, 2012.

\bibitem{LeeS14}
Y.~T. Lee and A.~Sidford.
\newblock Path finding methods for linear programming: Solving linear programs
  in {\~{o}}(sqrt(rank)) iterations and faster algorithms for maximum flow.
\newblock In {\em 2014 {IEEE} 55th Annual Symposium on Foundations of Computer
  Science}, pages 424--433, 2014.

\bibitem{mangasarian2006exact}
O.~L. Mangasarian.
\newblock Exact 1-norm support vector machines via unconstrained convex
  differentiable minimization.
\newblock {\em Journal of Machine Learning Research}, 7(Jul):1517--1530, 2006.

\bibitem{monteiro2013accelerated}
R.~D. Monteiro and B.~F. Svaiter.
\newblock An accelerated hybrid proximal extragradient method for convex
  optimization and its implications to second-order methods.
\newblock {\em SIAM Journal on Optimization}, 23(2):1092--1125, 2013.

\bibitem{nesterov2005smooth}
Y.~Nesterov.
\newblock Smooth minimization of non-smooth functions.
\newblock {\em Mathematical Programming}, 103(1):127--152, 2005.

\bibitem{nesterov2008accelerating}
Y.~Nesterov.
\newblock Accelerating the cubic regularization of {N}ewton’s method on
  convex problems.
\newblock {\em Mathematical Programming}, 112(1):159--181, 2008.

\bibitem{nesterov2013gradient}
Y.~Nesterov.
\newblock Gradient methods for minimizing composite functions.
\newblock {\em Mathematical Programming}, 140(1):125--161, 2013.

\bibitem{nesterov2018implementable}
Y.~Nesterov.
\newblock Implementable tensor methods in unconstrained convex optimization.
\newblock Technical report, Universit{\'e} catholique de Louvain, Center for
  Operations Research and Econometrics (CORE), 2018.

\bibitem{nesterov2006cubic}
Y.~Nesterov and B.~T. Polyak.
\newblock Cubic regularization of {N}ewton method and its global performance.
\newblock {\em Mathematical Programming}, 108(1):177--205, 2006.

\bibitem{orabona2012prisma}
F.~Orabona, A.~Argyriou, and N.~Srebro.
\newblock Prisma: Proximal iterative smoothing algorithm.
\newblock {\em arXiv preprint arXiv:1206.2372}, 2012.

\bibitem{platt1998sequential}
J.~Platt.
\newblock Sequential minimal optimization: A fast algorithm for training
  support vector machines.
\newblock Technical Report MSR-TR-98-14, April 1998.

\bibitem{shalev2011pegasos}
S.~Shalev-Shwartz, Y.~Singer, N.~Srebro, and A.~Cotter.
\newblock Pegasos: Primal estimated sub-gradient solver for svm.
\newblock {\em Mathematical Programming}, 127(1):3--30, 2011.

\bibitem{sherman2017area}
J.~Sherman.
\newblock Area-convexity, $\ell_\infty$ regularization, and undirected
  multicommodity flow.
\newblock In {\em Proceedings of the 49th Annual ACM SIGACT Symposium on Theory
  of Computing}, pages 452--460. ACM, 2017.

\bibitem{sidford2018coordinate}
A.~Sidford and K.~Tian.
\newblock Coordinate methods for accelerating $\ell_\infty$ regression and
  faster approximate maximum flow.
\newblock In {\em 2018 IEEE 59th Annual Symposium on Foundations of Computer
  Science}, pages 922--933, 2018.

\bibitem{spielman2004nearly}
D.~A. Spielman and S.-H. Teng.
\newblock Nearly-linear time algorithms for graph partitioning, graph
  sparsification, and solving linear systems.
\newblock In {\em Proceedings of the Thirty-Sixth Annual ACM Symposium on
  Theory of Computing}, pages 81--90. ACM, 2004.

\bibitem{woodworth2016tight}
B.~E. Woodworth and N.~Srebro.
\newblock Tight complexity bounds for optimizing composite objectives.
\newblock In {\em Advances in Neural Information Processing Systems}, pages
  3639--3647, 2016.

\bibitem{zhu2004onenorm}
J.~Zhu, S.~Rosset, R.~Tibshirani, and T.~J. Hastie.
\newblock 1-norm support vector machines.
\newblock In {\em Advances in Neural Information Processing Systems}, pages
  49--56, 2004.

\end{thebibliography}
\bibliographystyle{abbrv}

\newpage
\appendix
\section{Proofs}
\subsection{Proof of Lemma \ref{lem:smaxfourthbound}}

\begin{proof}[Proof of Lemma \ref{lem:smaxfourthbound}] 
It follows from \eqref{eq:smaxhessh1h2} that, for all $h \in \reals^m$, \begin{align}\label{eq:hesshhineq}
\hess \smax_\apm(x)[h,h] &= \frac{1}{\apm}\pa{\inner{\grad \smax_\apm(x), h^2} - \pa{\inner{\grad \smax_\apm(x), h}}^2}\nonumber\\
&\leq \frac{1}{\apm}\norm{\grad \smax_\apm(x)}_\infty \norm{h}^2 \leq \frac{1}{\apm}\norm{h}^2,
\end{align}
and
\begin{align}
\abs{\hess \smax_\apm(x)[h^2,h]} &\leq \frac{1}{\apm}\pa{\abs{\inner{\grad \smax_\apm(x), h^3}} + \abs{\pa{\inner{\grad \smax_\apm(x), h}}\pa{\inner{\grad \smax_\apm(x), h^2}}}}\nonumber\\ &\leq \frac{1}{\apm}\pa{\norm{\grad \smax_\apm(x)}_\infty \norm{h}_3^3 + \norm{\grad \smax_\apm(x)}_2\norm{\grad \smax_\apm(x)}_\infty\norm{h}_2^3}\nonumber\\
&\leq \frac{1}{\apm}\pa{\norm{\grad \smax_\apm(x)}_\infty + \norm{\grad \smax_\apm(x)}_2\norm{\grad \smax_\apm(x)}_\infty}\norm{h}_2^3\nonumber\\
&\leq \frac{2}{\apm}\norm{h}_2^3,\label{eq:hessh2hineq}
\end{align}
where the second and fourth inequalities follow from H\"older's inequality and \eqref{eq:gradnormineq}, respectively. We may similarly see that 
\begin{equation}\label{eq:hessh3hineq}
\abs{\hess \smax_\apm(x)[h^3,h]} \leq \frac{2}{\apm}\norm{h}_2^4.
\end{equation}

By taking the derivative of $\hess \smax_\apm(x)[h_1, h_2]$ with respect to $x$, for $h_1, h_2 \in \reals^m$, we have that
\begin{multline}\label{eq:thirdh1h2}
\third \smax_\apm(x)[h_1,h_2] = \frac{1}{\apm}(\hess \smax_\apm(x)[h_1\circ h_2] - \inner{\grad \smax_\apm(x), h_1}\hess \smax_\apm(x)[h_2] \\- \inner{\grad \smax_\apm(x), h_2}\hess \smax_\apm(x)[h_1]),
\end{multline}
and so for any $h \in \reals^m$,
\begin{equation}
\third \smax_\apm(x)[h,h] = \frac{1}{\apm}\pa{\hess \smax_\apm(x)[h^2] - 2\inner{\grad \smax_\apm(x), h}\hess \smax_\apm(x)[h]}.
\end{equation}
This implies that
\begin{equation}\label{eq:thirdsmaxhhh}
\third \smax_\apm(x)[h,h,h] = \frac{1}{\apm}\pa{\hess \smax_\apm(x)[h^2, h] - 2\inner{\grad \smax_\apm(x), h}\hess \smax_\apm(x)[h, h]},
\end{equation}
and so we may bound $\abs{\third \smax_\apm(x)[h,h,h]}$ by
\begin{align}
\abs{\third \smax_\apm(x)[h,h,h]} &\leq \frac{1}{\apm}\pa{\abs{\hess \smax_\apm(x)[h^2, h]} + 2\abs{\inner{\grad \smax_\apm(x), h}\hess \smax_\apm(x)[h, h]}}\nonumber\\
&\leq \frac{1}{\apm}\pa{\frac{2}{\apm}\norm{h}_2^3 + \frac{2}{\apm}\norm{h}_2^3} = \frac{4}{\apm^2}\norm{h}_2^3.\label{eq:thirdhhh}
\end{align}
Furthermore, since by \eqref{eq:thirdh1h2} we have that
\begin{align*}
\third \smax_\apm(x)[h^2,h] = \frac{1}{\apm}(\hess &\smax_\apm(x)[h^3] - \inner{\grad \smax_\apm(x), h^2}\hess \smax_\apm(x)[h] \\&- \inner{\grad \smax_\apm(x), h}\hess \smax_\apm(x)[h^2]),
\end{align*}
it follows that
\begin{align}
\bigl|\third \smax_\apm(x)[h^2,h,h]\bigr| &= \frac{1}{\apm}\bigl|\hess \smax_\apm(x)[h^3,h] - \inner{\grad \smax_\apm(x), h^2}\hess \smax_\apm(x)[h,h] \nonumber\\
&\qquad\qquad - \inner{\grad \smax_\apm(x), h}\hess \smax_\apm(x)[h^2,h]\bigr|\nonumber\\
&\leq \frac{1}{\apm}\bigl(\abs{\hess \smax_\apm(x)[h^3,h]} + \abs{\inner{\grad \smax_\apm(x), h^2}\hess \smax_\apm(x)[h,h]}\nonumber\\
&\qquad\qquad + \abs{\inner{\grad \smax_\apm(x), h}\hess \smax_\apm(x)[h^2,h]}\bigr)\nonumber\\
&\leq \frac{1}{\apm}\pa{\frac{2}{\apm}\norm{h}_2^4 + \frac{1}{\apm}\norm{h}_2^4 + \frac{2}{\apm}\norm{h}_2^4}\nonumber\\
&= \frac{5}{\apm^2}\norm{h}_2^4,\label{eq:thirdh2hh}
\end{align}
where the last inequality follows from \eqref{eq:hessh3hineq}, \eqref{eq:hesshhineq}, and combining \eqref{eq:gradnormineq} with another application of H\"older's inequality.

At this point, we may again take the derivative of $\third \smax_\apm(x)[h,h,h]$ (eq.\eqref{eq:thirdsmaxhhh}) with respect to $x$ to arrive at
\begin{align*}
\fourth \smax_\apm(x)[h,h,h] = \frac{1}{\apm}\bigl(\third &\smax_\apm(x)[h^2, h] - 2\inner{\grad \smax_\apm(x), h}\third \smax_\apm(x)[h,h] \\
&- 2\hess \smax_\apm(x)[h, h]\hess \smax_\apm(x)[h]\bigr).
\end{align*}
Note that, for all $h \in \reals^m$,
\begin{align*}
|\fourth \smax_\apm(x)[h,h,&h,h]| \leq \frac{1}{\apm}\biggl(\abs{\third \smax_\apm(x)[h^2, h, h]} + 2\abs{\inner{\grad \smax_\apm(x), h}\third \smax_\apm(x)[h,h,h]}\\
&\qquad\qquad\quad\quad + \abs{2\pa{\hess \smax_\apm(x)[h, h]}^2}\biggr)\\
& \leq \frac{1}{\apm}\biggl(\abs{\third \smax_\apm(x)[h^2, h, h]} + 2\norm{\grad \smax_\apm(x)}_2\norm{h}_2\abs{\third \smax_\apm(x)[h,h,h]}\\
&\qquad\qquad + \abs{2\pa{\hess \smax_\apm(x)[h, h]}^2}\biggr)\\
&\leq \frac{1}{\apm}\pa{\frac{5}{\apm^2}\norm{h}_2^4 + \frac{8}{\apm^2}\norm{h}_2^4 + \frac{2}{\apm^2}\norm{h}_2^4}\\
& = \frac{15}{\apm^3}\norm{h}_2^4,
\end{align*}
where the second inequality holds by H\"older's inequality, and the third inequality follows from \eqref{eq:thirdh2hh}, \eqref{eq:thirdhhh}, and \eqref{eq:hesshhineq}.
Thus, we arrive at the key result concerning the fourth derivative of $\smax_\apm(x)$, namely that, for all $x$, $h \in \reals^d$,
\begin{equation}\label{eq:smaxfourthbound}
\abs{\fourth \smax_\apm(x)[h,h,h,h]} \leq \frac{15}{\apm^3}\norm{h}_2^4.\qedhere
\end{equation}
\end{proof}

\subsection{Proof of Lemma \ref{lem:lipschitz}}
\begin{proof}[Proof of Lemma \ref{lem:lipschitz}]
Note that, for all $\zeta \in \reals^d$,
\begin{align*}
\norm{\fourth f(\zeta)}_{\bA^\top \bA}^* &= \max_{h : \norm{h}_{\bA^\top \bA} \leq 1} \abs{\fourth f(\zeta)[h,h,h,h]}\\
 &= \max_{h : \norm{\bA h}_2 \leq 1} \abs{\fourth f(\zeta)[h,h,h,h]}\\
 &\leq \max_{h : \norm{\bA h}_2 \leq 1} L_3\norm{\bA h}_2^4\\
 & = L_3,
\end{align*}
where the inequality follows from \eqref{eq:fourthboundsup}. Thus, we may see by a standard mean value theorem argument that, for any $x, y \in \reals^d$,
\begin{equation*}
\norm{\third f(y) - \third f(x)}_{\bA^\top \bA}^* \leq L_3\norm{y-x}_{\bA^\top \bA}.\qedhere
\end{equation*}
\end{proof}

\subsection{Proof of Theorem \ref{thm:mainlipschitz}}
\begin{proof}[Proof of Theorem \ref{thm:mainlipschitz}]
We may first note that, by the chain rule,
\begin{equation}\label{eq:hessaftertransform}
\grad f(x) = \bA^\top \grad \smax_\apm(\bA x-b) \quad \text{and} \quad \hess f(x) = \bA^\top \hess \smax_\apm(\bA x-b) \bA.
\end{equation}
Let $h \in \reals^d$. Then, we have that $\hess f(x)[h,h] = \hess \smax_\apm(\bA x-b)[\bA h,\bA h]$. By additional applications of the chain rule, it follows that 
\begin{equation}
\third f(x)[h,h] = \bA^\top \third \smax_\apm(\bA x-b)[\bA h,\bA h],
\end{equation}
 which implies that $\third f(x)[h,h,h] = \third \smax_\apm(\bA x-b)[\bA h,\bA h,\bA h]$. By one more chain rule application, we may see that
\begin{equation}
\fourth f(x)[h,h,h] = \bA^\top \fourth \smax_\apm(\bA x-b)[\bA h, \bA h, \bA h],
\end{equation}
which gives us
\begin{equation}
\fourth f(x)[h,h,h,h] = \fourth \smax_\apm(\bA x-b)[\bA h, \bA h, \bA h, \bA h].
\end{equation}
Combining this expression with \eqref{eq:smaxfourthbound}, we have that for all $h \in \reals^d$,
\begin{equation*}
\abs{\fourth f(x)[h,h,h,h]} \leq \frac{15}{\apm^3}\norm{\bA h}_2^4,
\end{equation*}
and so it follows from \ref{eq:highersmooth} and Lemma \ref{eq:fourthboundsup} that $f(x)$ is (order 3) $\frac{15}{\apm^3}$-smooth.
\end{proof}

\subsection{Proof of Corollary \ref{cor:mainlinfty}}
\begin{proof}[Proof of Corollary \ref{cor:mainlinfty}]
Note that, by Theorem \ref{thm:mainlipschitz}, we know that $f(x)$ is (order 3) $\frac{15}{\apm^3}$-smooth w.r.t. $\norm{\cdot}_{\bA^\top\bA}$, and so the iteration complexity follows. Part of the computational overhead of $\fastq$ comes from computing, for vectors $x, h_t \in \reals^d$,
\begin{equation}
c_t \defeq \grad f(x) + \hess f(x)h_t + \frac{1}{2}\third f(x)[h_t, h_t] + \frac{15}{\apm^3}\norm{h_t}_{\bA^\top\bA}\bA^\top \bA h_t,
\end{equation}
which can be done in time proportional to evaluating $f(x)$ \citep{agarwal2017finding, nesterov2018implementable}. In addition, for $\lambda > 0$, part of the (omitted) approximation procedure for each iteration of the method requires computing
\begin{equation}
g(\lambda) \defeq c_t^\top(\sqrt{2}\lambda \bA^\top \bA + \hess f(x))^{-1}\bA^\top \bA(\sqrt{2}\lambda \bA^\top \bA + \hess f(x))^{-1}c_t.
\end{equation}
Since we know by \eqref{eq:smaxhess} that $\hess f(x) = \bA^\top \diag(p(x))\bA - \bA^\top p(x)p(x)^\top\bA$, it follows that, for any $v \in \reals^d$,
\begin{align*}
(\sqrt{2}\lambda \bA^\top \bA + \hess f(x))^{-1}v &= ( \bA^\top(\diag(p(x))+\sqrt{2}\lambda\bI) \bA - \bA^\top p(x)p(x)^\top\bA)^{-1}v\\
&= \bM_x^{-1}v + \xi_{x,v}\bM_x^{-1}\bA^\top p(x)\\
&= \bM_x^{-1}(v+\xi_{x,v}\bA^\top p(x))
\end{align*}
where we let 
\begin{equation}
\bM_x \defeq (\sqrt{2}\lambda \bA^\top(\diag(p(x))+\bI) \bA), \qquad \xi_{x,v} \defeq \frac{p(x)^\top\bA \bM_x v }{1+p(x)^\top\bA \bM_x \bA^\top p(x)},
\end{equation}
and where the second equality follows from the Sherman–Morrison formula. Thus, the  computational bottleneck for each iteration of $\fastq$ is in solving $O(\log^{O(1)}(\Zcal/\eps))$ linear systems of the form
\begin{equation}
\bA^\top \bD_x \bA\phi = w_{x, v}
\end{equation}
for diagonal matrix $\bD_x \defeq (\diag(p(x))+\sqrt{2}\lambda\bI)$, and $w_{x, v} \defeq (v+\xi_{x,v}\bA^\top p(x))$.
\end{proof}

\subsection{Proof of Lemmas \ref{lem:l1approx} and \ref{lem:hingeapprox}}
\begin{proof}[Proof of Lemma \ref{lem:l1approx}]
Since $\abs{\scalar} = \max\braces{\scalar,-\scalar}$, it follows by Fact \eqref{fact:smaxapprox} that $\abs{\scalar} \leq \sabs_\apm(\scalar) \leq \abs{\scalar}+\apm$. Thus, we have that
\begin{equation}
\norm{x}_1 = \sum\limits_{i=1}^m \abs{x_i} \leq \sum\limits_{i=1}^m \sabs_\apm(x_i) \leq \norm{x}_1 + \apm m.
\end{equation}
\end{proof}

\begin{proof}[Proof of Lemma \ref{lem:hingeapprox}]
The proof follows immediately from Fact \eqref{fact:smaxapprox}.
\end{proof}

\subsection{Proof of Theorem \ref{thm:mainlipschitzsvm}}
\begin{proof}[Proof of Theorem \ref{thm:mainlipschitzsvm}]
First, we observe that since $\softlone_\apm(x) = \sum\limits_{i=1}^d \sabs_\apm(x_i)$, it follows that
\begin{equation}
    \abs{\fourth \softlone_\apm(x)[h,h,h,h]} \leq \sum\limits_{i=1}^d \abs{\fourth \sabs_\apm(x_i)[h,h,h,h]} \leq \frac{15d}{\apm^3}\norm{h}_2^4.
\end{equation}
In addition, we may note that
\begin{equation*}
    \abs{\fourth \ssvm_\apm(\tQ x)[\tQ h,\tQ h,\tQ h,\tQ h]} \leq \frac{1}{m}\sum\limits_{i=1}^m \frac{15}{\apm^3}\norm{\tQ h}_2^4 \leq \frac{15\norm{\tQ^\top \tQ}^2}{\apm^3}\norm{h}_2^4.
\end{equation*}

Taken together, we may see that
\begin{align*}
    \abs{\fourth f_\apm(x)[h,h,h,h]} &= \abs{\lambda\fourth \softlone_\apm(x)[h,h,h,h] + \fourth \ssvm_\apm(\tQ x)[\tQ h, \tQ h, \tQ h, \tQ h]}\\
    &\leq \lambda\abs{\fourth \softlone_\apm(x)[h,h,h,h]} + \abs{\fourth \ssvm_\apm(\tQ x)[\tQ h, \tQ h, \tQ h, \tQ h]}\\
    &\leq \frac{15\pa{\lambda d + \norm{\tQ^\top \tQ}^2}}{\apm^3}\norm{h}_2^4,
\end{align*}
and so the theorem follows from Lemma \ref{lem:lipschitz}.
\end{proof}

\subsection{Proof of Corollary \ref{cor:mainl4svm}}
\begin{proof}[Proof of Corollary \ref{cor:mainl4svm}]
Let $g(x) = \lambda\norm{x}_4^4$. First, we may note that $g(x)$ is (order 4) $O(\frac{1}{d})$-uniformly convex w.r.t. $\norm{\cdot}_2$ (see, e.g., Lemma 2.4 in \cite{bullins2018fast}). The proof follows by combining Theorem \ref{thm:kappa} with the proof of Theorem \ref{thm:mainlipschitzsvm}, in addition to the observation that, for all $x, h \in \reals^d$, \begin{equation*}
\abs{\fourth g(x)[h,h,h,h]} \leq 24\lambda\norm{h}_2^4.\qedhere
\end{equation*}
\end{proof}

\end{document}